\documentclass[a4paper, 12pt]{article}

\usepackage{amsfonts}
\usepackage {amssymb}
\usepackage {amsmath}
\usepackage {amsthm}
\usepackage{graphicx}
\usepackage{multirow}
\usepackage{xypic}
\usepackage {amscd}
\usepackage{mathrsfs}
\usepackage[colorlinks, linkcolor=blue, anchorcolor=black, citecolor=red]{hyperref}
\usepackage{enumerate}
\usepackage{enumitem}
\usepackage{geometry}  
\usepackage{tikz-cd}
\usepackage[all]{xy}

\usepackage{titlesec}

\newtheorem{thm}{Theorem}[section]
\newtheorem{prop}[thm]{Proposition}
\newtheorem{defn}[thm]{Definition}

\newtheorem{cor}[thm]{Corollary}
\newtheorem{rem}[thm]{Remark}

\newtheorem{lem}[thm]{Lemma}

\newtheorem{defn-prop}[thm]{Definition-Proposition}

\newtheorem{obs}{Observation}

\titleformat{\part}
  {\Large\bfseries}
  {\partname\ \thepart}
  {1em}
  {}

\titlespacing*{\part}{0pt}{3ex}{2ex}

\begin{document}

\title{Blaschke Conjecture and Complex Geometry}
\author{Kyobeom Song}
\date{}

\maketitle

\begin{abstract}
We provide a complex-geometric approach to the Blaschke conjecture, i.e., that a manifold whose injectivity radius equals its diameter is isometric to a compact rank-one symmetric space (CROSS). In particular, we introduce the great quadric fibration, a complex analogue of the great sphere fibration. Requiring its total space to be a complex submanifold simultaneously explains many properties that a Blaschke manifold is expected to possess, including its Clifford structure, Hopf fibration, and diffeomorphism class. Furthermore, the great quadric bundle coincides with the variety of minimal rational tangents (VMRT) of the ambient Fano variety in the standard CROSS cases, and we explain this through a bend-and-break argument in the setting of the Blaschke conjecture under a suitable complexification assumption. Combining this with Hwang--Mok VMRT recognition, we prove the Blaschke conjecture for manifolds admitting an adapted complex structure on their entire tangent bundles.
\end{abstract}
\section{Introduction}
A compact, connected Riemannian manifold \((M,g)\) is called:
\begin{enumerate}
\item \emph{Blaschke} if its injectivity radius at every point equals its diameter.
\item \emph{Zoll} if all unit-speed geodesics are periodic with the same least period.
\end{enumerate}
Every Blaschke manifold is known to be Zoll \cite[Corollary~1]{M13}, but the converse does not hold in general. In 1921, Wilhelm Blaschke posed the question of whether a Blaschke surface in \(\mathbb R^3\) must be a round sphere \cite[\S86, Question~2]{B21}. This question was subsequently generalized as follows \cite{M13}.

\noindent\textbf{Blaschke conjecture.} Every Blaschke manifold is isometric to a compact rank-one symmetric space (CROSS).

What is remarkable about the Blaschke conjecture is that the CROSSes, namely, \(S^n\), \(\mathbb{RP}^n\), \(\mathbb{CP}^n\), \(\mathbb{HP}^n\), and \(\mathbb{OP}^2\), possess highly sophisticated Clifford, Lie-theoretic, algebraic, and geometric structures, different even from one another, yet all of these are characterized by the single condition that the injectivity radius equals the diameter. Indeed, the Blaschke conjecture has long been approached from many different directions in geometry, bringing numerous fields into interplay. For example, Morse theory \cite{B54}, geometric analysis \cite{W74,BK80,K82,Y80}, and symplectic and differential topology \cite{GWY83,M05,KS07} have all contributed to its study. For further references, we recommend the classical book of Besse \cite{B12} and the modern survey of McKay \cite{M13}.

We propose a complex-geometric approach to the Blaschke conjecture and great sphere fibrations, based on the following key observation.

\begin{obs}\label{hopfvmrt}
The Hopf fibrations of standard CROSSes \(G/K\) and the varieties of minimal rational tangents (VMRTs) of the natural equivariant compactifications of \(G^{\mathbb C}/K^{\mathbb C}\) are two manifestations of the same geometric data.
\end{obs}

We now describe this correspondence explicitly. Let \(\mathbb K=\mathbb C,\mathbb H,\mathbb O\), with \(a=2,4,8\), respectively, and \(n=2\) in the octonionic case. Write \(E:=\mathbb R^{an}\), equipped with its standard Euclidean inner product \(g\). A smooth fibration of the unit sphere \(S^{an-1}\subset E\) over a smooth manifold \(B\) of real dimension \(an-a\), whose fibers are great \((a-1)\)-spheres, is called a \emph{great sphere fibration}. The standard examples are the Hopf fibrations, with \(B=\mathbb KP^{n-1}\):
\[
S^{an-1}\subset (E,g)=(\mathbb R^{an},g),\qquad S^{a-1}\longrightarrow S^{an-1}\xrightarrow{\rho}B,\qquad S^{a-1}\longrightarrow S^{an-1}\longrightarrow\mathbb KP^{n-1}.
\]
In particular, for a point \(p\) in a Blaschke manifold, the map \(\rho:U_pM\to B_p\) sending each unit vector \(v\in U_pM\) to the cut point reached by the geodesic with initial velocity \(v\) is known to be a great sphere fibration; for a CROSS, this is the Hopf fibration. In fact, a Blaschke manifold whose great sphere fibration is isomorphic to the Hopf fibration is known to be homeomorphic to the corresponding CROSS. However, McKay \cite[p.~20, Section~20]{M13} points out that this structure forgets too much geometric information for its analysis to readily yield isometric rigidity.

We suggest, however, that considering its complex analogue may yield more information, possibly even an isometry. For a great sphere fibration over \(B\), consider the real tautological bundle \(\mathcal V\) of rank \(a\), obtained by taking the linear span of each fiber \(S^{a-1}\subset E\), together with its complexification \(\mathcal V^{\mathbb C}\):
\begin{align*}
\mathbb R^a&\longrightarrow \mathcal V\xrightarrow{\rho}B, & \mathcal V&\subset B\times E,\\
\mathbb C^a&\longrightarrow \mathcal V^{\mathbb C}\xrightarrow{\rho}B, & \mathcal V^{\mathbb C}&\subset B\times E^{\mathbb C}.
\end{align*}
Since the great sphere fibration consists of the unit vectors in \(\mathcal V\), a natural complex analogue is obtained by collecting the isotropic lines in each fiber \(\mathbb P(\mathcal V_q^{\mathbb C})\simeq\mathbb{CP}^{a-1}\). Thus we have the following analogy:
\begin{align*}
S^{a-1}&\longrightarrow S^{an-1}\xrightarrow{\rho}B, & g(v,v)&=1\quad(v\in S^{an-1}), & S^{an-1}&\subset E,\\
Q^{a-2}&\longrightarrow Z\xrightarrow{\rho}B, & g_{\mathbb C}(\xi,\xi)&=0\quad([\xi]\in Z), & Z&\subset\mathbb P(E^{\mathbb C}).
\end{align*}
Here \(g_{\mathbb C}\) is the complex-bilinear extension of \(g\). We call \(Q^{a-2}\to Z\to B\) the \emph{great quadric fibration} and its total space \(Z\) the \emph{great quadric bundle}.

Here, the great quadric bundle reveals the correspondence in Observation~\ref{hopfvmrt}. Unlike the total space of the Hopf fibration, which is simply a sphere \(S^{an-1}\), the total space of the great quadric fibration is rather nontrivial. These spaces are listed in Table \ref{amb}, together with the compactification of the tangent bundle \(TM\) of each CROSS \(M\) into a Fano variety \(X:=TM\sqcup D\) \cite{PW91}. For \(S^n\), the entry is obtained by regarding the exponential map from the unit tangent sphere at \(p\in S^n\) to its antipodal point \(q\) as a bundle over the singleton \(\{q\}\).
\begin{table}[htbp]
\centering
\small
\begin{tabular}{c|ccccc}
\(M\) & \(S^n\) & \(\mathbb{RP}^n\) & \(\mathbb{CP}^n\) & \(\mathbb{HP}^n\) & \(\mathbb{OP}^2\) \\ \hline
\(X:=TM\sqcup D\) & \(Q^n\) & \(\mathbb{CP}^n\) & \(\mathbb{CP}^n\times\mathbb{CP}^n\) & \(\operatorname{Gr}(2,2n+2)\) & \(E_6/P_1\) \\
\(D\) & \(Q^{n-1}\) & \(Q^{n-1}\) & \(\{z\cdot w=0\}\) & \(\operatorname{IG}(2,2n+2)\) & \(F_4(\mathbb C)/P_4\) \\
\(Z\) & \(Q^{n-2}\) & Undefined & \(\mathbb P^{n-1}\sqcup\mathbb P^{n-1}\) & \(\mathbb P^1\times\mathbb P^{2n-1}\) & \(\mathbb S_5\)
\end{tabular}
\caption{Great quadric bundles and ambient complex spaces of CROSSes.}\label{amb}
\end{table}

Remarkably, except in the case of \(\mathbb{RP}^n\), these are the varieties of minimal rational tangents (VMRT) of the ambient Fano varieties \(X\). Moreover, under the identification \(T_pM^{\mathbb C}\simeq T_pX\) induced by the totally real embedding \(M\subset X\), \(Z\) is precisely the VMRT of \(X\) at \(p\). This is no coincidence, including the fact that \(Z\) is undefined in the \(\mathbb{RP}^n\) case; see Section~\ref{vmrt} for a more detailed discussion.

Unlike the bundles arising from Hopf fibrations, however, a general great quadric bundle \(Z\subset\mathbb P(E^{\mathbb C})\) is a priori only a smooth submanifold. Nevertheless, holomorphicity of this embedding alone suffices to obtain all of the following rigidity properties.

\medskip
\noindent\textbf{Summary of Theorem~\ref{hopfiso} and Corollary~\ref{sb}.}
If a great quadric bundle \(Z\subset\mathbb P(E^{\mathbb C})\) is a complex submanifold, then:
\begin{enumerate}
\item \(Z\) is projectively equivalent to the standard model;
\item the great sphere fibration is orthogonally isomorphic to the standard Hopf fibration;
\item the associated Clifford structures are recovered through the same principle;
\item Moreover, if this hypothesis holds for \(Z_p\) at even a single point \(p\) of a Blaschke manifold \(M\), then \(M\) is diffeomorphic to a CROSS.
\end{enumerate}

The proof identifies the complex-geometric structure of \(Z\subset\mathbb P(E^{\mathbb C})\) with the corresponding Clifford-algebraic data of the Hopf fibration, whose rigidity then yields the above conclusions. Thus, several rigidity properties expected of great sphere fibrations and Blaschke manifolds admit a unified explanation through a single principle: the holomorphicity of the great quadric bundle.

Furthermore, the standardness of the VMRT of \(X\) determines the biholomorphism type of \(X\) itself \cite{M08}, and there are results relating the biholomorphic rigidity of the ambient Fano variety \(X\) of a Blaschke manifold \(M\) to the isometric rigidity of \(M\) \cite{BL18,S25,LS26}. This strongly suggests a connection between the standardness of the great quadric bundle \(Z_p\) and the metric rigidity of \(M\). Indeed, we prove the following.
\begin{thm}\label{main}
Every Blaschke manifold admitting an adapted complex structure on the entire tangent bundle is isometric to a compact rank-one symmetric space.
\end{thm}
An adapted complex structure is, roughly, a structure in which every geodesic extends to a complex curve in the tangent bundle; see Section~\ref{bacs} for the precise definition. In fact, we first apply a bend-and-break argument to the extension \(\phi_\gamma\) of a geodesic \(\gamma\) as a curve to show that \(Z_p\) is the VMRT at each point \(p\in M\). This allows us to apply Theorem~\ref{hopfiso}, which, combined with Hwang--Mok recognition, yields the desired conclusion.

Finally, we remark that the assumptions of Theorems~\ref{hopfiso} and~\ref{main} are very strong. In the proposed complex-geometric approach to the Blaschke conjecture, the essential difficulty lies in constructing complexifications satisfying the assumptions of these theorems. Nevertheless, their proofs give a unified explanation, through the holomorphicity of great quadric bundles, of various rigidity phenomena associated with CROSSes, including metric rigidity. Thus, our motivation for introducing great quadric fibrations is not only to provide a route to resolving the Blaschke conjecture, but also to provide a coherent viewpoint that illuminates the mechanisms underlying these rigidity phenomena.

\noindent\textbf{Note added.} On October 6, 2026, OpenAI released a preprint \cite{OAI26} claiming a proof of the Blaschke conjecture. The preprint cites the first version of the present paper.

\noindent \textbf{Acknowledgement.}~The author is deeply grateful to Professor Chi Li for discussions that greatly helped bring this paper to completion. The author was partially supported by the NSF under Award Nos. DMS-2305296 and DMS-2529637.

\noindent \textbf{AI Disclosure.}~ChatGPT 5.6 Sol suggested that the author's original proof of assertion~(1) of Theorem~\ref{hopfiso} could be reformulated in terms of Clifford representations, making the proof more structured and leading to assertion~(3) of the same theorem. ChatGPT 6 Astra suggested using a cross-ratio argument to prove non-contraction in the bend-and-break construction in Proposition~\ref{break}. ChatGPT 6 Astra was used for language editing, LaTeX typesetting, and translation of selected passages, with all text initially written by the author.

\section{Great quadric fibration}
In this section, we give the precise definition of the great quadric fibration and explain why it can be interpreted as the VMRT in the standard CROSS cases.
\subsection{Definitions and conventions}
By the Bott--Samelson theorem \cite{B12}, for every Blaschke or Zoll manifold \((M,g)\), there exist positive integers \(a\) and \(n\) such that \(\dim_{\mathbb R}M=an\) and every closed geodesic \(\gamma\) has Morse index \(\operatorname{ind}(\gamma)=a-1\). Moreover, the cohomology ring of \(M\) is determined by \(a\) and \(n\), and the possibilities are as follows:
\[
\begin{array}{c|ccccc}
(a,n) & (a,1) & (1,n) & (2,n) & (4,n) & (8,2) \\ \hline
H^*(M) & H^*(S^a) & H^*(\mathbb{RP}^n) & H^*(\mathbb{CP}^n) & H^*(\mathbb{HP}^n) & H^*(\mathbb{OP}^2)
\end{array}
\]
Here, following a convention similar to that of Allamigeon--Warner \cite{A65,B12}, we adopt the following terminology.
\begin{defn}
A Blaschke manifold, respectively a Zoll manifold, of real dimension \(an\) for which every closed geodesic has Morse index \(a-1\) is called a \emph{Blaschke manifold of type \((a,n)\)}, respectively a \emph{Zoll manifold of type \((a,n)\)}.
\end{defn}

Let $(M,g)$ be a Blaschke manifold of type $(a,n)$, and fix $p\in M$. Each unit vector $v\in T_pM$ determines a unique point $\rho(v)$ in the cut locus $B_p$, namely, the first point at which the geodesic with initial velocity $v$ meets $B_p$. The resulting map satisfies the following definition.

\begin{defn}\label{gsf}
Let $S^{an-1}\subset\mathbb R^{an}$ be the standard unit sphere, and let $B$ be a smooth manifold without boundary of real dimension $an-a$. Suppose that:
\begin{enumerate}
\item $\rho:S^{an-1}\to B$ is a smooth surjective submersion;
\item for every $q\in B$, there exists a real $a$-dimensional linear subspace $V_q\subset\mathbb R^{an}$ such that $\rho^{-1}(q)=S^{an-1}\cap V_q\simeq S^{a-1}$.
\end{enumerate}
Then $S^{a-1}\longrightarrow S^{an-1}\xrightarrow{\rho}B$ is called a \emph{great sphere fibration}.
\end{defn}

The pairs $(a,n)$ for which such fibrations exist are known to be precisely the Blaschke types listed above. We now consider the following tautological extension.

\begin{defn}\label{rb}
Let $S^{a-1}\longrightarrow S^{an-1}\xrightarrow{\rho}B$ be a great sphere fibration. Set $E:=\mathbb R^{an}$, equipped with its standard Euclidean inner product $g$, and let $\underline E:=B\times E$ be the trivial vector bundle. Define the tautological subbundle $\mathcal V\subset\underline E$ by taking the linear span of each fiber of the great sphere fibration, so that $\mathcal V=\{(q,v)\in B\times E:v\in V_q\}$. We again denote its bundle projection by $\rho$, writing $\mathbb R^a\longrightarrow\mathcal V\xrightarrow{\rho}B$.

Let $\mathcal W:=\mathcal V^\perp$ be the fiberwise orthogonal complement of $\mathcal V$ with respect to $g$. We call $\mathcal V$ the \emph{real tautological bundle} and the exact sequence of real vector bundles $0\to\mathcal V\to\underline E\to\mathcal W\to0$, whose last map is orthogonal projection, the \emph{real tautological sequence}.
\end{defn}

For a Blaschke manifold $(M,g)$ and a point $p\in M$, we take $(E,g)=(T_pM,g_p)$ and denote the corresponding bundles by $\mathcal V_p$ and $\mathcal W_p$. We next consider their complexifications.

\begin{defn}\label{cb}
Let $\mathbb R^a\longrightarrow\mathcal V\xrightarrow{\rho}B$ be a real tautological bundle, and consider the fiberwise complexifications
\[
\mathcal V^{\mathbb C}:=\mathcal V\otimes_{\mathbb R}\mathbb C,\qquad
\underline E^{\mathbb C}:=\underline E\otimes_{\mathbb R}\mathbb C,\qquad
\mathcal W^{\mathbb C}:=\mathcal W\otimes_{\mathbb R}\mathbb C.
\]
Let $g_{\mathbb C}$ be the complex-bilinear extension of $g$, given by $g_{\mathbb C}(u+iv,x+iy)=g(u,x)-g(v,y)+i\bigl(g(u,y)+g(v,x)\bigr)$. For $a\geq2$, define
\[
Z:=\left\{(q,[\xi])\in\mathbb P(\mathcal V^{\mathbb C}):g_{\mathbb C}(\xi,\xi)=0\right\},\qquad Q^{a-2}\longrightarrow Z\xrightarrow{\rho}B.
\]
We call this the \emph{great quadric fibration} associated with the great sphere fibration, and its total space $Z$ the \emph{great quadric bundle}.
\end{defn}

Indeed, \(Z\) is smooth, since the fiberwise quadratic form \(g_{\mathbb C}(z,z)\) has nonzero vertical differential away from the zero section along its zero locus. Moreover, the natural map \(Z\to\mathbb P(E^{\mathbb C})\) is a smooth embedding: its base point is recovered locally by \(q=\rho(\operatorname{Re}z/\|\operatorname{Re}z\|)\), for a local smooth representative \(z\), so it is an injective immersion, and \(Z\) is compact.

For a Blaschke manifold, we similarly denote the corresponding bundles associated with a point $p\in M$ by $\mathcal V^\mathbb C_p$ and $Z_p$.

\subsection{Great quadric fibrations of CROSSes}\label{vmrt}
We now explain why the great quadric bundle \(Z_p\) associated with a point $p\in M$ of a standard CROSS should be the VMRT, and why \(\mathbb{RP}^n\) is an exception. This is a rough explanation; for a precise proof that also applies to the general case, see Section~\ref{bacs}.

As shown in Table~\ref{amb}, the tangent bundle \(TM\) of a CROSS \(M\) admits a natural compactification into a Fano manifold \(X\). Moreover, every geodesic \(\gamma:S^1\to M\) is closed and extends to a rational curve given in cylindrical coordinates \(z=\sigma+i\tau\), with \(\sigma\in\mathbb R/2\pi\mathbb Z\), by
\[
\phi_\gamma:\mathbb P^1\to X,\qquad \phi_\gamma(\sigma+i\tau)=\tau\gamma'(\sigma)\in TM.
\]
At the two ends \(\tau\to\pm\infty\), the map \(\phi_\gamma\) is extended to take values in \(D=X\setminus TM\). Thus every closed geodesic \(\gamma\) extends to a rational curve \(\phi_\gamma\) meeting the ample divisor \(D\) transversely at two points.

The metric \(g\) of the CROSS \(M\) also extends to a holomorphic Riemannian metric \(g_{\mathbb C}\) on \(TM\subset X\) \cite{K14}. In particular, if \(s_D\) denotes the defining section of \(D\), then \(s_D^2g_{\mathbb C}:TX^{\otimes2}\to\mathcal O_X(2D)\) is holomorphic. Hence, along a smooth rational curve \(C\) with \(D\cdot C=1\), restriction to the tangent directions gives
\[
(s_D^2g_{\mathbb C})|_{TC^{\otimes2}}:TC^{\otimes2}\simeq\mathcal O_{\mathbb P^1}(4)\longrightarrow\mathcal O_{\mathbb P^1}(2D\cdot C)\simeq\mathcal O_{\mathbb P^1}(2).
\]
This map is therefore zero. In particular, the tangent directions of a minimal rational curve \(C\) are \(g_{\mathbb C}\)-isotropic wherever \(g_{\mathbb C}\) is defined.

Now consider the cases in Table~\ref{amb} other than \(\mathbb{RP}^n\). Note that \(D\) is a primitive ample divisor on \(X\). Since \(D\cdot\phi_\gamma=2\), the curve \(\phi_\gamma\) is a conic. For antipodal points \(p,q\in M\), the tangent directions at \(p\) of all such curves \(\phi_\gamma\) passing through \(p\) and \(q\) contain \(\mathbb P_{\mathbb R}((\mathcal V_p)_q)\simeq\mathbb{RP}^{a-1}\), the real locus of the projective linear subspace \(\mathbb P((\mathcal V_p^{\mathbb C})_q)\simeq\mathbb P^{a-1}\). Thus the Zariski closure of the tangent directions at \(p\) of all conics through \(p\) and \(q\) is \(\mathbb P((\mathcal V_p^{\mathbb C})_q)\simeq\mathbb P^{a-1}\). The \(g_{\mathbb C}\)-isotropic directions \(Q^{a-2}\subset\mathbb P^{a-1}\) should then be precisely the directions arising from broken conics, namely the tangent directions of minimal rational curves. Finally, fixing \(p\) and varying the antipodal point \(q\), the union of these quadrics \(Q^{a-2}\) is exactly the great quadric bundle \(Z_p\).

For \(\mathbb{RP}^n\), the divisor \(D=Q^{n-1}\) in \(X=\mathbb{CP}^n\) has degree \(2\). Thus \(\phi_\gamma\), which satisfies \(D\cdot\phi_\gamma=2\), is itself a minimal rational curve. Consequently, the VMRT should be the full projective space \(\mathbb P(T_pX)\simeq\mathbb P^{n-1}\), obtained by complexifying the set of tangent directions of the curves \(\phi_\gamma\), rather than \(Z_p\).
\begin{rem}
Note that the crucial point in the above explanation is that every geodesic \(\gamma\) extends to a rational curve \(\phi_\gamma\). Thus, as long as this property persists in the general case, essentially the same mechanism can be applied. This observation will be used in the proof of Theorem~\ref{main}.
\end{rem}

\section{Rigidity of great quadric bundles}
We now prove the following first main theorem.

\begin{thm}\label{hopfiso}
Consider the great quadric fibration \(Q^{a-2}\longrightarrow Z\xrightarrow{\rho}B\), and suppose that \(Z\subset\mathbb P(E^\mathbb C)\) is a complex submanifold. Let \(\mathbb K=\mathbb C,\mathbb H,\mathbb O\) according as \(a=2,4,8\), and $\dim_\mathbb R E=an$.
\begin{enumerate}
\item The projective equivalence class of \(Z\) is given by
\[
\begin{array}{c|ccc}
(a,n) & (2,n) & (4,n) & (8,2) \\ \hline
Z & \mathbb P^{n-1}\sqcup\mathbb P^{n-1} & \mathbb P^1\times\mathbb P^{2n-1} & \mathbb S_5
\end{array}
\]
\item The real tautological bundle \(\mathcal V\subset B\times E\) and its associated great sphere fibration \(S(V_q)\longrightarrow S(E)\xrightarrow{\rho}B\), where \(V_q:=\mathcal V_q\), are orthogonally isomorphic to the standard real rank-\(a\) Hopf bundle \(\gamma_{\mathbb K}\subset\mathbb KP^{n-1}\times\mathbb K^n\) and the standard Hopf fibration \(S(\mathbb K)\longrightarrow S(\mathbb K^n)\xrightarrow{\pi}\mathbb KP^{n-1}\), respectively. More precisely, there exist a linear isometry \(\phi:(E,g)\to\mathbb K^n\), a diffeomorphism \(\varphi:B\to\mathbb KP^{n-1}\), and a vector bundle isomorphism \(\widetilde\phi:\mathcal V\to\gamma_{\mathbb K}\), given by \(\widetilde\phi(q,v)=(\varphi(q),\phi(v))\), such that the following diagrams commute:
\[
\begin{tikzcd}[column sep=small]
\mathcal V \arrow[r,hook] \arrow[d,"\widetilde\phi"']
& B\times E \arrow[d,"\varphi\times\phi"] \\
\gamma_{\mathbb K} \arrow[r,hook]
& \mathbb KP^{n-1}\times\mathbb K^n
\end{tikzcd}
\qquad
\begin{tikzcd}[column sep=small]
S(E) \arrow[r,"\rho"] \arrow[d,"\phi"']
& B \arrow[d,"\varphi"] \\
S(\mathbb K^n) \arrow[r,"\pi"']
& \mathbb KP^{n-1}
\end{tikzcd}
\]
\item For each fiber \(V_q\) of the real tautological bundle, equipped with the metric \(g\), there exists an even Clifford representation \(c_q:\operatorname{Cl}^0(V_q,g)\to\operatorname{End}_{\mathbb R}(T_qB)\) with the following property. Let \(Q:=\rho^{-1}(q)\simeq Q^{a-2}\). Under the real vector bundle identification \(N^{\mathbb R}_{Q/Z}\simeq Q\times T_qB\), the complex structure of the holomorphic normal bundle \(N_{Q/Z}\) defines a family \(J_q:Q\to\operatorname{End}_{\mathbb R}(T_qB)\), and \(J_q([\alpha+i\beta])=c_q(\alpha\beta)\) whenever \(\alpha,\beta\in V_q\) are orthonormal.
\end{enumerate}
\end{thm}

\noindent\textbf{Summary of proof.}
To characterize the projective variety \(Z\subset\mathbb P(E^{\mathbb C})\), we analyze its topology and anticanonical bundle. The former is obtained by elementary algebraic topology. For the latter, we study the holomorphic normal bundle \(N_{Q/Z}\) of a fiber \(Q\simeq Q^{a-2}\). In doing so, we show that its complex structure satisfies Clifford relations, proving assertion~3. The resulting Clifford action restricts the possible normal bundles and determines the restriction of the anticanonical bundle to each fiber. We then apply characterization theorems for projective varieties to prove assertion~1; the case \(a=2\) is treated by a direct degree calculation. Finally, using representation theory and linear algebra, we show that the projective equivalence of \(Z\) to the standard model already encodes the full structure of the standard Hopf fibration, proving assertion~2.

\medskip
\noindent\textbf{Topological background.}
If \(n=1\), then \(B\) is a point and \(Z\) is the standard quadric \(Q^{a-2}\), so the assertions are immediate. Thus assume that \(n\geq2\). The homotopy exact sequence \(\pi_1(S^{an-1})\longrightarrow\pi_1(B)\longrightarrow\pi_0(S^{a-1})\) gives \(\pi_1(B)=0\), so \(B\) is simply connected and orientable. The real tautological bundle \(\mathcal V\) is also orientable. The Gysin sequence of its sphere bundle \(S^{a-1}\longrightarrow S^{an-1}\longrightarrow B\) gives \(H^*(B;\mathbb Z)\simeq\mathbb Z[\varepsilon]/(\varepsilon^n)\), where \(\varepsilon=e(\mathcal V)\) has degree \(a\). In particular, the integral cohomology of \(B\) is torsion-free and concentrated in even degrees. The same holds for \(Q^{a-2}\), so the integral Serre spectral sequence of the great quadric fibration \(Q^{a-2}\longrightarrow Z\xrightarrow{\rho}B\) degenerates at \(E_2\) for degree reasons, yielding an additive isomorphism \(H^*(Z;\mathbb Z)\simeq H^*(B;\mathbb Z)\otimes H^*(Q^{a-2};\mathbb Z)\).

\medskip
\noindent\textbf{The normal bundle \(N_{Q/Z}\).}
The key step is to compute the restriction of the canonical bundle of \(Z\) to a fiber \(Q^{a-2}\). For this purpose, we analyze the normal bundle of a fixed fiber. Its underlying real bundle is trivial, namely, \(N^{\mathbb R}_{Q/Z}\simeq Q\times T_qB\). However, the complex structure on each fiber varies over \(Q\), giving a nontrivial holomorphic structure. We therefore analyze these complex structures along a fixed fiber. Fix \(q\in B\), and set
\[
Q:=\rho^{-1}(q),\qquad V:=\mathcal V_q,\qquad W:=V^{\perp_g}\subset E.
\]
Let \(V^{\mathbb C}\) and \(W^{\mathbb C}\) denote their complexifications.

Consider the following diagram of bundle maps:
\[
\begin{tikzcd}
N_{Q/Z} \arrow[r,"d\rho^{1,0}"] \arrow[dr,"\Psi"']
& Q\times T_qB^{\mathbb C} \arrow[d,"\operatorname{ev}\circ dG_q^{\mathbb C}"] \\
& \mathcal O_Q(1)\otimes W^{\mathbb C}.
\end{tikzcd}
\]
Its components are as follows:
\begin{itemize}
\item \(N_{Q/Z}:=(T^{1,0}Z|_Q)/T^{1,0}Q\) is the holomorphic normal bundle of \(Q\) in \(Z\).
\item \(d\rho^{1,0}:N_{Q/Z}\to Q\times T_qB^{\mathbb C}\) is defined by
\[
v^{1,0}:=\frac{v-iJv}{2},\qquad d\rho^{1,0}([v^{1,0}])=\frac{d\rho(v)-id\rho(Jv)}{2}.
\]
\item \(G\) is the classifying map of the complex tautological sequence,
\[
G:B\longrightarrow\operatorname{Gr}(a,E^{\mathbb C}),\qquad G(q')=\mathcal V_{q'}^{\mathbb C}.
\]
Its differential at \(q\) is
\[
dG_q:T_qB\longrightarrow\operatorname{Hom}_{\mathbb C}(V^{\mathbb C},E^{\mathbb C}/V^{\mathbb C})\simeq\operatorname{Hom}_{\mathbb C}(V^{\mathbb C},W^{\mathbb C}),
\]
and \(dG_q^{\mathbb C}\) denotes its complex-linear extension.
\item \(\mathcal O_Q(1):=\mathcal O_{\mathbb P(E^{\mathbb C})}(1)|_Q\), and \(\operatorname{ev}\) is defined fiberwise by
\[
\operatorname{ev}_{[z]}:\operatorname{Hom}_{\mathbb C}(V^{\mathbb C},W^{\mathbb C})\longrightarrow\operatorname{Hom}_{\mathbb C}(\mathbb Cz,W^{\mathbb C})=(\mathcal O_Q(1)\otimes W^{\mathbb C})_{[z]},\qquad A\longmapsto A|_{\mathbb Cz}.
\]
\item \(\Psi\) is the map induced by the inclusion \(Z\subset\mathbb P(E^{\mathbb C})\) and projection modulo \(T\mathbb P(V^{\mathbb C})|_Q\):
\[
\Psi:N_{Q/Z}\hookrightarrow N_{\mathbb P(V^{\mathbb C})/\mathbb P(E^{\mathbb C})}|_Q\simeq\bigl(\mathcal O_{\mathbb P(V^{\mathbb C})}(1)\otimes(E^{\mathbb C}/V^{\mathbb C})\bigr)|_Q\simeq\mathcal O_Q(1)\otimes W^{\mathbb C}.
\]
\end{itemize}
This diagram commutes by definition. More explicitly, fix \(x=[z]\in Q\), with \(z\in V^{\mathbb C}\), and consider a curve \(X(t)=[Z(t)]\in Z\) with \(Z(0)=z\). Let \(a\in N_{Q/Z,x}\) be the normal class of \(X'(0)\), and set
\[
b:=Z'(0)\in T_zE^{\mathbb C},\qquad \Phi:=\operatorname{ev}\circ dG_q^{\mathbb C}\circ d\rho^{1,0}.
\]
By definition, the identification of \(dG_q\) with an element of \(\operatorname{Hom}_{\mathbb C}(V^{\mathbb C},E^{\mathbb C}/V^{\mathbb C})\) sends a vector \(z\in V^{\mathbb C}\) to the vector \(b\in T_zE^{\mathbb C}\) induced by the variation of \(V^{\mathbb C}\) in the Grassmannian, modulo motion in the \(V^{\mathbb C}\)-direction. Thus \(\Phi(x,a)=(x,f)\), where \(f:\mathbb Cz\to W^{\mathbb C}\) satisfies \(f(z)=[Z'(0)]=[b]\in E^{\mathbb C}/V^{\mathbb C}\simeq W^{\mathbb C}\). The map \(\Psi\) likewise assigns a vector in \(W^{\mathbb C}\) using the lifted curve \(Z(t)\) and its tangent vector \(b\) at a point of the affine cone over \(Q\). Hence the two constructions agree.

\medskip
\noindent\textbf{Proof of assertion 3.}
For \(z\in V^{\mathbb C}\), define
\[
C_z:T_qB^{\mathbb C}\longrightarrow W^{\mathbb C},\qquad C_z(\xi):=dG_q^{\mathbb C}(\xi)(z).
\]
By definition, \(z\mapsto C_z\) is linear. For every nonzero real vector \(v\in V\), the map \(C_v\) is a linear isomorphism. Indeed, the evaluation map \(S(\mathcal V)\to S(E)\), given by \((q',u)\mapsto u\), is a diffeomorphism with inverse \(u\mapsto(\rho(u),u)\). Locally near \(q\), write \(\mathcal V_{q'}\) as the graph of a linear map \(A(q'):V\to W\), with \(A(q)=0\). Then \(dA_q(\xi)v=C_v(\xi)\). For a unit vector \(v\in V\), the differential of the evaluation map, in the resulting local coordinates, is \((\xi,\eta)\mapsto C_v(\xi)+\eta\), where \(\xi\in T_qB\) and \(\eta\in V\cap v^\perp\). Since \(T_vS(E)=(V\cap v^\perp)\oplus W\), the restriction \(C_v:T_qB\to W\) is an isomorphism. Complexification and linearity in \(v\) give the claim for every nonzero real \(v\).

Let \(x=[\alpha+i\beta]\in Q\), with \(\alpha,\beta\in V\), and let \(\xi\in T_qB\). Denote by \(J_x\) the complex structure transported to \(T_qB\) through the preceding diagram. The holomorphicity of \(\Psi\) gives
\[
\begin{aligned}
0&=dG_q^{\mathbb C}(\xi+iJ_x\xi)(\alpha+i\beta)=C_{\alpha+i\beta}(\xi+iJ_x\xi)\\
&=C_\alpha\xi-C_\beta J_x\xi+i(C_\beta\xi+C_\alpha J_x\xi).
\end{aligned}
\]
Thus \(C_\alpha-C_\beta J_x=0\) and \(C_\beta+C_\alpha J_x=0\). Since \(J_x\) is a complex structure, we obtain the following key identities:
\begin{equation}\label{J}
J_x=C_\beta^{-1}C_\alpha=-C_\alpha^{-1}C_\beta,\qquad J_x^2=-\operatorname{id}.
\end{equation}
Fix a real unit vector \(e\in V\), and set
\[
e_0:=e,\qquad L_z:=C_e^{-1}C_z\in\operatorname{End}_{\mathbb C}(T_qB^{\mathbb C}).
\]
Define \(c_q:\operatorname{Cl}^0(V,g)\to\operatorname{End}_{\mathbb R}(T_qB)\) by \(c_q(e_0z):=-L_z\) for \(z\in V\).

Equation~\eqref{J} recovers the Clifford relations as follows. For an orthonormal basis \(e_0,e_1,\ldots,e_{a-1}\) of \(V\), write \(C_i:=C_{e_i}\) and \(L_i:=L_{e_i}\). For distinct \(i,j\in\{1,\ldots,a-1\}\), we have
\[
\begin{aligned}
L_iL_j&=C_0^{-1}C_iC_0^{-1}C_j=-C_i^{-1}C_0C_0^{-1}C_j=-C_i^{-1}C_j=C_j^{-1}C_i,\\
L_jL_i&=C_0^{-1}C_jC_0^{-1}C_i=-C_j^{-1}C_0C_0^{-1}C_i=-C_j^{-1}C_i.
\end{aligned}
\]
Hence \(L_iL_j=-L_jL_i\), and the universal property yields a Clifford representation \(c_q\). Moreover, for orthonormal \(\alpha,\beta\in V\), the Clifford identity \(\alpha\beta=-(e_0\alpha)^{-1}(e_0\beta)\) and the definition \(c_q(e_0z)=-L_z\) give
\[
c_q(\alpha\beta)
=-L_\alpha^{-1}L_\beta
=-C_\alpha^{-1}C_\beta
=J_{[\alpha+i\beta]}.
\]

\medskip
\noindent\textbf{Proof of assertion 1.}
This Clifford representation strongly constrains the form of \(N_{Q/Z}\). Note that the image bundle of \(\Psi\) is isomorphic to \(N_{Q/Z}\). To examine this image more explicitly, define the composite map
\[
\begin{aligned}
\mu:\mathcal O_Q(-1)\otimes T_qB^{\mathbb C}&\longrightarrow Q\times W^{\mathbb C}\longrightarrow Q\times T_qB^{\mathbb C},\\
z\otimes\xi&\longmapsto([z],C_z(\xi))\longmapsto([z],L_z(\xi)).
\end{aligned}
\]
By the definition of the first map and the fact that \(C_0\) in \(L_z=C_0^{-1}C_z\) is a linear isomorphism, we obtain \(\mathcal O_Q(1)\otimes\operatorname{Im}(\mu)\simeq N_{Q/Z}\). The variation of \(L_z(\xi)\) with \(z\) is the restriction of the Clifford action associated with \(V^{\mathbb C}\) to isotropic vectors, which makes this map highly rigid.

\medskip
\noindent\textbf{1. \((a,n)=(4,n)\), \(\mathbb K=\mathbb H\).}

In this case, the complexified Clifford algebra is \((\mathbb H\oplus\mathbb H)\otimes_{\mathbb R}\mathbb C\simeq M_2(\mathbb C)\oplus M_2(\mathbb C)\), and the corresponding module is \(\mathbb H^{\mathbb C}\simeq M_2(\mathbb C)\). More explicitly, choose an orthonormal basis \(e_0=e,e_1,e_2,e_3\) of \(V\), and set \(I=\left(\begin{smallmatrix}i&0\\0&-i\end{smallmatrix}\right)\), \(J=\left(\begin{smallmatrix}0&1\\-1&0\end{smallmatrix}\right)\), and \(K=\left(\begin{smallmatrix}0&i\\i&0\end{smallmatrix}\right)\). Identify \(T_qB^{\mathbb C}\simeq M_2(\mathbb C)^{\oplus(n-1)}\) so that \(L_{e_1}\) and \(L_{e_2}\) act on each summand by left multiplication by \(I\) and \(J\), respectively. Let \(k_\pm\) be the numbers of summands on which \(L_{e_3}\) acts by left multiplication by \(\pm K\). Thus \(k_++k_-=n-1\), and \(T_qB^{\mathbb C}\simeq M_2(\mathbb C)^{\oplus k_+}\oplus M_2(\mathbb C)^{\oplus k_-}\).

Under the identification of \(V^{\mathbb C}\) with \(M_2(\mathbb C)\) given by the basis \(\operatorname{id},I,J,K\), we have \(g_{\mathbb C}(z,z)=\det(z)\). Thus nonzero isotropic vectors in \(V^{\mathbb C}\) correspond to rank-one complex matrices \(A=vw^{\mathsf t}\), with \(v,w\in\mathbb C^2\). On each summand \(M_2(\mathbb C)\) of \(T_qB^{\mathbb C}\), the line bundles arising from the action of \(A=vw^{\mathsf t}\) depend only on \(v\) or only on \(w\). Each of the two columns of \(M_2(\mathbb C)\) contributes one line bundle, so \(\operatorname{Im}(\mu)\simeq\mathcal O(-1,0)^{\oplus2k_+}\oplus\mathcal O(0,-1)^{\oplus2k_-}\) and \(N_{Q/Z}\simeq\mathcal O(1,0)^{\oplus2k_-}\oplus\mathcal O(0,1)^{\oplus2k_+}\). In particular, \(\det N_{Q/Z}\simeq\mathcal O(2k_-,2k_+)\). Since \(-K_Q\simeq\mathcal O(2,2)\), adjunction gives
\[
(-K_Z)|_Q\simeq(-K_Q)\otimes\det N_{Q/Z}\simeq\mathcal O(2k_-+2,2k_++2).
\]
We now use the following characterization theorem of Occhetta \cite[Theorem~1.1]{O06}.
\begin{thm}[Occhetta]
Let \(X\) be a smooth complex projective variety of dimension \(N\). Then \(X\simeq\mathbb P^{n_1}\times\cdots\times\mathbb P^{n_k}\) if and only if there exist \(k\) unsplit covering families of rational curves \(\mathcal V_1,\ldots,\mathcal V_k\) such that \(-K_X\cdot\mathcal V_i=n_i+1\) and \(n_1+\cdots+n_k=N\), and the numerical classes \([\mathcal V_1],\ldots,[\mathcal V_k]\in N_1(X)\) are linearly independent.
\end{thm}
In our situation, the two rulings of \(Q^2\simeq\mathbb P^1\times\mathbb P^1\) can be globally distinguished because \(B\) is simply connected. By taking curves in one ruling or the other in every fiber, we obtain two families of rational curves on \(Z\), say \(\mathcal V_1,\mathcal V_2\). They cover \(Z\), and they are unsplit because each family is parametrized by a compact family of lines in the fibers.

If \(C_1\) and \(C_2\) denote curves in the two rulings, then the above adjunction computation gives \(-K_Z\cdot C_1=2k_-+2\) and \(-K_Z\cdot C_2=2k_++2\). Moreover, the Serre spectral sequence shows that \([C_1]\) and \([C_2]\) are linearly independent in \(H_2(Z,\mathbb Q)\). Since \(Z\) is projective and \(b_2(Z)=2\), Hodge decomposition gives \(H^{2,0}(Z)=0\), so \(H^2(Z,\mathbb Q)\) is generated by divisor classes by the Lefschetz \((1,1)\)-theorem. Thus \([C_1]\) and \([C_2]\) are also numerically independent. Therefore Occhetta's theorem gives \(Z\simeq\mathbb P^{2k_-+1}\times\mathbb P^{2k_++1}\).

Finally, the additive cohomology of \(Z\), already determined from the fibration \(Q^2\longrightarrow Z\longrightarrow B\), agrees with that of the canonical VMRT. Among the above products of two projective spaces, the only possibility is $Z\simeq\mathbb P^1\times\mathbb P^{2n-1}$. The restriction map \(H^2(Z,\mathbb Z)\to H^2(Q^2,\mathbb Z)\) is an isomorphism, and the ample line bundle \(\mathcal O_Z(1)\) has degree one on both rulings of \(Q^2\); hence \(\mathcal O_Z(1)\simeq\mathcal O(1,1)\). Moreover, the isotropic vectors in the fibers of \(\mathcal V^{\mathbb C}\) span \(E^{\mathbb C}\), so \(Z\) spans \(\mathbb P(E^{\mathbb C})\). Since \(h^0(\mathcal O(1,1))=4n=\dim_{\mathbb C}E^{\mathbb C}\), the embedding is given by the complete linear system \(|\mathcal O(1,1)|\), and \(Z\) is projectively equivalent to the Segre variety.

\medskip
\noindent\textbf{2. \((a,n)=(8,2)\), \(\mathbb K=\mathbb O\).}

In this case, the representation is equivalent to the action of \(\operatorname{Cl}^0(V,g)\simeq\operatorname{Cl}_{0,7}\) on \(T_qB\simeq\mathbb O\). Thus, if \(\Delta_\pm\) are the complexifications of the irreducible real half-spin modules \(S_\pm\), then, up to fixed linear isomorphisms, \(\mu\) is induced by Clifford multiplication \(V^{\mathbb C}\otimes\Delta_\pm\to\Delta_\mp\). Moreover, by the holomorphicity of \(\mu\), the \((0,1)\)-part of \(T_qB^{\mathbb C}\) is annihilated, so \(\operatorname{Im}(\mu)\) is a holomorphic bundle of rank \(4\). Hence \(\operatorname{Im}(\mu)\) is one of Ottaviani's two rank-\(4\) spinor bundles \(\mathcal S_\pm\) on \(Q^6\) \cite{O88}. By \cite[Remark~2.9]{O88}, \(\det(\mathcal S_\pm)\simeq\mathcal O_Q(-2)\). Consequently, \(\det N_{Q/Z}\simeq\mathcal O_Q(1)^{\otimes4}\otimes\det(\mathcal S_\pm)\simeq\mathcal O_Q(2)\), and hence \(c_1(N_{Q/Z})=2H\), where \(H:=c_1(\mathcal O_Q(1))\). Since \(c_1(-K_Q)=6H\) for \(Q=Q^6\), adjunction gives
\[
c_1((-K_Z)|_Q)=c_1(-K_Q)+c_1(N_{Q/Z})=8H.
\]
Moreover, \(\dim_{\mathbb C}Z=\dim_{\mathbb C}Q+\operatorname{rank}_{\mathbb C}N_{Q/Z}=6+4=10\). The great quadric fibration \(Q^6\longrightarrow Z\longrightarrow B\) identifies the second cohomology groups of \(Q^6\) and \(Z\), so \(Z\subset\mathbb P(E^{\mathbb C})\simeq\mathbb P^{15}\) is a \(10\)-dimensional Fano manifold of Picard number \(1\) and index \(8\), hence of coindex \(3\).

Each fiber \(Q^6\) is the projectivization of the isotropic vectors in a fiber of the complex tautological bundle \(\mathcal V^{\mathbb C}\subset B\times E^{\mathbb C}\), and these isotropic vectors span that fiber. As the base point varies, the real tautological bundle sweeps out all directions in \(E\), so \(Z\) spans \(\mathbb P^{15}\). Thus the restriction map \(H^0(\mathbb P^{15},\mathcal O(1))\to H^0(Z,\mathcal O_Z(1))\) is injective, and \(h^0(Z,\mathcal O_Z(1))\geq16\). The Mukai genus is therefore \(g=h^0(Z,\mathcal O_Z(1))-\dim_{\mathbb C}Z+1\geq7\). By Mukai's classification of Fano manifolds of coindex \(3\) \cite{M89,W21}, \(Z\subset\mathbb P^{15}\) is projectively equivalent to \(\mathbb S_5\).

\medskip
\noindent\textbf{3. \((a,n)=(2,n)\), \(\mathbb K=\mathbb C\).}

In this case, \(Q^{a-2}=Q^0=\{[v^{1,0}],[v^{0,1}]\}\) is a two-point set, so the first statement of Theorem~\ref{hopfiso} admits a more straightforward proof by the following degree calculation than by an analysis of \(N_{Q/Z}\) analogous to the cases \(a=4,8\). Since the base \(B\) of the great quadric fibration satisfies \(\pi_1(B)=0\), we have \(Z=Z_+\sqcup Z_-\), and each \(Z_\pm\) is diffeomorphic to \(B\) via the bundle projection \(\rho\). Set \(L:=\mathcal O_{\mathbb P(E^{\mathbb C})}(-1)|_{Z_+}\). Then the sphere bundle \(S(L)\) is equivalent to the great sphere fibration as follows. For a unit vector \(v\in\mathcal V_q\subset E\), let \(J\) be the orthogonal complex structure on \(\mathcal V_q\) for which \([v-iJv]\in Z_+\) (note that \(\dim_{\mathbb R}\mathcal V_q=2\)). The following correspondence follows directly from the definitions:
\[
\begin{tikzcd}
S(L) \arrow[r] \arrow[d,"\simeq"']
& Z_+ \arrow[d,"\rho"] \\
S^{2n-1} \arrow[r,"\rho"']
& B
\end{tikzcd}
\qquad
\begin{tikzcd}
\frac{1}{\sqrt{2}}(v-iJv) \arrow[r,mapsto] \arrow[d,mapsto]
& {[v-iJv]} \arrow[d,mapsto] \\
v \arrow[r,mapsto]
& q.
\end{tikzcd}
\]
Let \(H:=c_1(L^*)=-c_1(L)\), and let \(e\) be the Euler class of \(S(L)\). Then \(e=\pm H\), depending on the sign convention. The Gysin sequence of the great sphere fibration \(S^1\longrightarrow S^{2n-1}\longrightarrow B\) therefore gives \(\left|\int_{Z_+}e^{n-1}\right|=\left|\int_{Z_+}H^{n-1}\right|=1\). Since \(H\) is ample, \(\int_{Z_+}H^{n-1}=1\). Thus \(Z_+\) is a projective subvariety of degree \(1\) in \(\mathbb P(E^{\mathbb C})\), hence a linear subspace. The same argument applies to \(Z_-\). Since \(\dim_{\mathbb R}B=2n-2\), we have \(\dim_{\mathbb C}Z_\pm=n-1\). These are disjoint linear subspaces of \(\mathbb P^{2n-1}\), so their dimensions imply that \(Z=Z_+\sqcup Z_-\subset\mathbb P(E^{\mathbb C})\) is projectively equivalent to \(\mathbb P^{n-1}\sqcup\mathbb P^{n-1}\subset\mathbb P^{2n-1}\).

\medskip
\noindent\textbf{Proof of assertion 2.}
Let \(Z_0\subset\mathbb P((\mathbb K^n)^{\mathbb C})\) be the great quadric bundle of the standard model, and let \(g_0\) be the standard Euclidean metric on \(\mathbb K^n\). It suffices to construct a real-linear isomorphism \(\phi\) satisfying
\[
\phi:E\xrightarrow{\sim}\mathbb K^n,\qquad \phi^*g_0=g,\qquad \mathbb P(\phi^{\mathbb C})(Z)=Z_0.
\]
Indeed, \(Z\) determines which orthonormal pairs lie in a common fiber: two orthonormal real vectors \(a,b\in E\) belong to the same fiber of \(\mathcal V\) if and only if \([a+ib]\in Z\). Since \(\phi\) preserves the metric and identifies \(Z\) with \(Z_0\), it follows that \(\rho(a)=\rho(b)\) if and only if \(\pi(\phi(a))=\pi(\phi(b))\). The same equivalence holds for arbitrary real unit vectors: for a linearly independent pair, replace the second vector by its normalized component orthogonal to the first, while a linearly dependent pair belongs to the same fiber on both sides. Thus \(\phi\) maps the fibers of the original great sphere fibration bijectively onto those of the standard Hopf fibration.

Consequently, \(\phi\) induces a bijection \(\varphi\) on the bases, characterized by
\[
\begin{tikzcd}[column sep=large]
S(E) \arrow[r,"\phi"] \arrow[d,"\rho"']
& S(\mathbb K^n) \arrow[d,"\pi"] \\
B \arrow[r,"\varphi"']
& \mathbb KP^{n-1}
\end{tikzcd}
\qquad
\varphi(q):=\pi(\phi(v))\quad\text{for any }v\in S(\mathcal V_q).
\]
Since \(\rho\) is a smooth submersion, it has smooth local sections, and \(\varphi\circ\rho=\pi\circ\phi\) shows that \(\varphi\) is smooth. Applying the same argument to \(\varphi^{-1}\circ\pi=\rho\circ\phi^{-1}\) shows that its inverse is smooth. Thus \(\varphi\) is a diffeomorphism. Since \(\phi\) is linear and maps each sphere fiber onto the corresponding Hopf sphere fiber, it also maps their linear spans onto one another. Therefore \(\widetilde\phi(q,v):=(\varphi(q),\phi(v))\) defines the required vector bundle isomorphism \(\mathcal V\to\gamma_{\mathbb K}\), and both diagrams in the statement commute.

We now construct \(\phi:(E,g)\to\mathbb K^n\) in the following steps. Write \(h:=(g_0)_{\mathbb C}\) for the complex-bilinear extension of \(g_0\). By the first statement of Theorem~\ref{hopfiso}, there exists a complex-linear isomorphism \(f:E^{\mathbb C}\to(\mathbb K^n)^{\mathbb C}\) whose projectivization satisfies \(\mathbb P(f)(Z)=Z_0\). We first prove the following lemma.
\begin{lem}
Set \(I_2:=\{b\in\operatorname{Sym}^2(((\mathbb K^n)^{\mathbb C})^*):b(z,z)=0\text{ for every }[z]\in Z_0\}\). The complex-linear automorphisms of \((\mathbb K^n)^{\mathbb C}\) whose projectivizations preserve \(Z_0\) act transitively on the nondegenerate forms in \(I_2\).
\end{lem}
\begin{proof}
For \(\mathbb K=\mathbb C\), we have \(Z_0=\mathbb P(A_+)\sqcup\mathbb P(A_-)\), where \(\dim_{\mathbb C}A_\pm=n\). In this case, \((\mathbb C^n)^{\mathbb C}=A_+\oplus A_-\), where the left-hand side denotes the complexification of the \emph{real} vector space \(\mathbb C^n\). Thus \(b\in I_2\) has block matrix \(\left(\begin{smallmatrix}0&C\\C^{\mathsf t}&0\end{smallmatrix}\right)\), and \(b\) is nondegenerate precisely when \(C\) is invertible. Changes of basis in the two summands act transitively on such matrices \(C\) and preserve \(Z_0\).

For \(\mathbb K=\mathbb H\), we have \(Z_0=\mathbb P(A)\times\mathbb P(W)\), where \(\dim_{\mathbb C}A=2\) and \(\dim_{\mathbb C}W=2n\). Here \(I_2=\bigwedge^2A^*\otimes\bigwedge^2W^*\). Since \(\dim_{\mathbb C}A=2\), fixing a nonzero \(\omega_A\in\bigwedge^2A^*\) gives a unique expression \(b=\omega_A\otimes\omega_W\) for every \(b\in I_2\). The form \(b\) is nondegenerate if and only if \(\omega_W\) is nondegenerate, so the nondegenerate forms in \(I_2\) constitute a single orbit under \(\operatorname{id}_A\otimes\operatorname{GL}(W)\).

For \(\mathbb K=\mathbb O\), we have \(Z_0=\mathbb S_5\). Let \(\rho_{\mathrm{spin}}:\operatorname{Spin}(10,\mathbb C)\to\operatorname{GL}((\mathbb O^2)^{\mathbb C})\simeq\operatorname{GL}(16,\mathbb C)\) be the half-spin representation. The action on \(I_2\) defined by \((s\cdot b)(z,z):=b(\rho_{\mathrm{spin}}(s)^{-1}z,\rho_{\mathrm{spin}}(s)^{-1}z)\) is isomorphic to the standard representation of \(\operatorname{Spin}(10,\mathbb C)\) \cite[Corollary~4.6 and Lemma~6.20]{K18}. Let \(Q\) denote the invariant quadratic form of this standard representation. We claim that \(Q(b)=0\) implies that \(b\) is degenerate.

Indeed, \(\operatorname{Spin}(10,\mathbb C)\) has two orbits on \(\mathbb P(I_2)\simeq\mathbb P^9\), namely, \(\mathcal Q:=\{[b]:Q(b)=0\}\) and its complement. Set \(\mathcal D:=\{[b]\in\mathbb P(I_2):\det b=0\}\). We show that \(\mathcal D=\mathcal Q\). First, \(\mathcal D\) is closed, since it is defined by the determinant, and is \(\operatorname{Spin}(10,\mathbb C)\)-invariant, since the action preserves degeneracy. Moreover, \(\mathcal D\) is proper and nonempty. The complexification \(h\) of the standard metric gives a nondegenerate form with \([h]\in\mathbb P(I_2)\setminus\mathcal D\). Restricting the determinant to a projective line through \([h]\) gives a nonzero homogeneous polynomial of degree \(16\), which has a zero. Hence \(\mathcal D\neq\varnothing\). It follows that \(\mathcal D\) is precisely the closed orbit \(\mathcal Q\).

Consequently, every nondegenerate \(b\in I_2\) satisfies \(Q(b)\neq0\). The standard action of \(\operatorname{Spin}(10,\mathbb C)\) is transitive on vectors with the same nonzero value of \(Q\). Thus any two nondegenerate forms \(b,b'\in I_2\) are related by a combination of this action and scalar multiplication. Scalar multiplication of forms is induced by scalar automorphisms of \((\mathbb O^2)^{\mathbb C}\), completing the proof.
\end{proof}
It follows that there exists a complex-linear automorphism \(A\in\operatorname{GL}((\mathbb K^n)^{\mathbb C})\) preserving \(Z_0\) such that \(F:=Af:(E^{\mathbb C},g_{\mathbb C})\to((\mathbb K^n)^{\mathbb C},h)\) is an isometry and \(\mathbb P(F)(Z)=Z_0\). We finally modify \(F\) to obtain a real isometry.

Choose real orthonormal bases in the domain and codomain, so that the matrix of \(F\) satisfies \(F^{\mathsf t}F=I\). Take its polar decomposition \(F=UP\). Since \(F^{\mathsf t}F=I\), we have \(\overline F=(F^*)^{-1}=UP^{-1}\). On the other hand, conjugating \(F=UP\) gives \(\overline F=\overline U\,\overline P\). Uniqueness of the polar decomposition therefore implies \(\overline U=U\) and \(\overline P=P^{-1}\). In particular, \(U\) is real orthogonal. Since both \(Z\) and \(Z_0\) are invariant under conjugation, suppressing projectivizations in the notation, we obtain
\[
F(Z)=\overline F(Z)=Z_0,\qquad UP(Z)=UP^{-1}(Z),\qquad P^2(Z)=Z.
\]
We now apply the following lemma.
\begin{lem}
Let \(P\) be a positive-definite Hermitian matrix and \(Z\) a projective algebraic subset. If \(P^2(Z)=Z\), then \(P(Z)=Z\).
\end{lem}
\begin{proof}
For a homogeneous polynomial \(f\), set \(T(f):=f\circ P\). Choose an orthonormal eigenbasis \(e_1,\ldots,e_N\) of \(P\), with \(Pe_i=\lambda_i e_i\) and \(\lambda_i>0\). Using a Hermitian inner product that is complex-linear in the first argument, define coordinates \(x_i(v):=\langle v,e_i\rangle\). Let \(R_d\) be the space of homogeneous polynomials of degree \(d\), and let \(I_d(Z)\subset R_d\) be the subspace of those vanishing on \(Z\). Declare the monomials \(x^\alpha\) to be an orthonormal basis. Then \(T(x^\alpha)=(\prod_i\lambda_i^{\alpha_i})x^\alpha\), so \(T\) is represented by a positive real diagonal matrix.

Since \(P^2(Z)=Z\), the subspace \(I_d(Z)\) is \(T^2\)-invariant. An invariant subspace of a diagonalizable operator is the direct sum of its intersections with the eigenspaces. Since all eigenvalues of \(T\) are positive, \(T\) and \(T^2\) have the same eigenspaces, and \(T\) acts by a scalar on each of them. Hence \(I_d(Z)\) is also \(T\)-invariant. Since \(T\) is invertible, \(T(I_d(Z))=I_d(Z)\). This holds for every \(d\), so \(P(Z)=Z\).
\end{proof}
Thus \(P(Z)=Z\), and consequently \(U(Z)=UP(Z)=F(Z)=Z_0\). The linear map \(\phi:(E,g)\to\mathbb K^n\) represented by the real orthogonal matrix \(U\) therefore has the required properties.\qed \medskip

If the great sphere fibration \(S^{a-1}\to S^{an-1}\xrightarrow{\rho}B_p\) at a point \(p\) of a Blaschke manifold is topologically isomorphic to the standard Hopf fibration, then the manifold is homeomorphic to its model CROSS \cite[Corollary~2]{M13}. Assertion~2 of Theorem~\ref{hopfiso} gives the additional information that this identification is induced by a linear isometry \(\phi:(E,g)\to(\mathbb K^n,h)\) of the ambient real vector spaces. This yields the following stronger conclusion.

\begin{cor}\label{sb}
Let \((M,g)\) be a Blaschke manifold. If the great quadric bundle \(Z_p\subset\mathbb P(T_pM^{\mathbb C})\) is a complex submanifold for some \(p\in M\), then \(M\) is diffeomorphic to its model CROSS.
\end{cor}

\begin{proof}
Since the Blaschke conjecture is known for \(S^n\) and \(\mathbb{RP}^n\), it suffices to consider type \((a,n)=(2,n),(4,n),(8,2)\), with \(\mathbb K=\mathbb C,\mathbb H,\mathbb O\), respectively. Let \((M_0,h)\) be the model CROSS, scaled so that both manifolds have diameter \(D\). Fix \(p_0\in M_0\), and write \(B:=B_p\) and \(B_0:=B_{p_0}\simeq\mathbb KP^{n-1}\). Identifying \(T_{p_0}M_0\) with \(\mathbb K^n\), assertion~2 of Theorem~\ref{hopfiso} provides a linear isometry \(\phi:(T_pM,g_p)\to(T_{p_0}M_0,h_{p_0})\) and a diffeomorphism \(\varphi:B\to B_0\) satisfying \(\varphi\circ\rho=\rho_0\circ\phi\). Define a bijection \(F:M\to M_0\) by
\[
F(x):=
\begin{cases}
\exp_{p_0}\!\bigl(\phi(\exp_p^{-1}(x))\bigr),&x\in M\setminus B,\\
\varphi(x),&x\in B,
\end{cases}
\]
where \(\exp_p^{-1}\) takes values in the open ball of radius \(D\). The map \(F\) restricts to a diffeomorphism \(M\setminus B\to M_0\setminus B_0\). We show that it is also a local diffeomorphism along \(B\).

Let \(\mathcal V\to B\) be the real tautological bundle, so that \(S(\mathcal V_q)=\rho^{-1}(q)\). For \(v\in S(\mathcal V_q)\), set \(\gamma_v(t):=\exp_p(tv)\). We claim that the endpoint maps extend to a smooth vector bundle isometry
\[
A:(\mathcal V,g_p)\xrightarrow{\sim}(N_{B/M},g),
\qquad
A_qv=-\dot\gamma_v(D)\quad\text{for }v\in S(\mathcal V_q).
\]
Assuming this claim, let \(A^0\) denote the corresponding isometry for \(M_0\). Then
\[
\widehat\phi_q
:=A^0_{\varphi(q)}\circ\phi|_{\mathcal V_q}\circ A_q^{-1}
:N_qB\longrightarrow N_{\varphi(q)}B_0
\]
defines a smooth vector bundle isometry covering \(\varphi\). For \(v\in S(\mathcal V_q)\) and sufficiently small \(s\geq0\), reversing the geodesic gives
\[
\exp_q(sA_qv)=\exp_p((D-s)v),
\qquad
F\bigl(\exp_q(sA_qv)\bigr)
=\exp^0_{\varphi(q)}\bigl(sA^0_{\varphi(q)}\phi(v)\bigr).
\]
Thus, in normal exponential coordinates, \(F\) is the bundle isomorphism \((q,w)\mapsto(\varphi(q),\widehat\phi_qw)\). Consequently, \(F\) is a local diffeomorphism along \(B\), and hence a global diffeomorphism.

It remains to prove the claim. Fix \(q\in B\) and \(u,v\in S(\mathcal V_q)\). Set \(p_s:=\gamma_u(s)\) and \(q_s:=\gamma_u(D+s)\). The Blaschke condition gives \(d(p_s,q_s)=D\). Choose a smooth variation \(c_s\) of \(\gamma_v\) with endpoints \(p_s,q_s\). Since \(L(c_s)\geq D=L(c_0)\), the first variation formula yields
\[
0=\left.\frac{d}{ds}\right|_{s=0}L(c_s)
=g_q\bigl(\dot\gamma_u(D),\dot\gamma_v(D)\bigr)-g_p(u,v).
\]
The endpoint velocities are normal to \(B\), so the map \(v\mapsto-\dot\gamma_v(D)\) preserves inner products from \(S(\mathcal V_q)\) to \(S(N_qB)\). Since both vector spaces have dimension \(a\), it extends uniquely to a linear isometry \(A_q:\mathcal V_q\to N_qB\). Smooth dependence on \(q\) follows by evaluating \(A_q\) on local smooth orthonormal frames of \(\mathcal V\), using the smoothness of the geodesic flow. This proves the claim.
\end{proof}

\begin{rem}[Relation to the Sato map]
The use of the differential of the classifying map to construct linear operators on \(T_qB_p\) in the proof of Theorem~\ref{hopfiso} has a precedent in McKay's construction of the osculating complex structure, an important tool for understanding the topological Blaschke conjecture in the \(\mathbb{CP}^n\) case \cite{M16}.

In the \(\mathbb{CP}^n\) case, choose an orientation of the real Blaschke bundle and let \(j_q\) be the corresponding orthogonal complex structure on \(V_q\). The maps
\[
\sigma_\pm:B_p\longrightarrow\mathbb{P}(T_pM^{\mathbb C}),
\qquad
\sigma_\pm(q)=[v\mp i j_qv],
\quad 0\ne v\in V_q,
\]
are well defined and satisfy
\[
Z_p=\sigma_+(B_p)\sqcup\sigma_-(B_p).
\]
These are metric-compatible versions of McKay's Sato-type construction: the fiberwise complex structures are selected using \(g_p\), rather than the osculating construction. Under the hypotheses of Theorem~\ref{hopfiso}, the great circle fibration is orthogonally isomorphic to the standard Hopf fibration, so the two choices agree \cite[Section~11, Lemma~13]{M16}. Thus \(Z_p\) is then the union of McKay's Sato image and its complex conjugate.

The great quadric construction therefore gives an intrinsic formulation of these metric-compatible Sato-type images in terms of isotropic lines. It requires no auxiliary choice of fiberwise complex structures and extends naturally to the \(\mathbb{HP}^n\) and \(\mathbb{OP}^2\) cases.  
\end{rem}

\section{Blaschke Manifolds with Adapted Complex Structure}\label{bacs}
We now prove Theorem~\ref{main} using Theorem~\ref{hopfiso}\,\textup{(1)} and VMRT theory. This suggests a possible connection between the rigidity of great quadric bundles and isometric rigidity of Blaschke manifolds.

\subsection{Zoll manifolds with entire Grauert tubes}
Every Blaschke manifold is known to be Zoll, meaning that all geodesics are periodic with the same period \cite[Corollary~1]{M13}. We first give the precise definition of an adapted complex structure and describe the structures of Zoll manifolds admitting one. For Grauert tube theory, see \cite{LS91,S91,GS91}; for more detailed references specific to Zoll metrics, see \cite{BL18,LS26,S25}. For a classical reference on Zoll manifolds themselves, see \cite{B12}.

An adapted complex structure is defined as follows.
\begin{defn}
Let \((M,g)\) be a compact Riemannian manifold. A complex structure \(J\) on the total space \(TM\) of the tangent bundle is called an \emph{adapted complex structure} if, for every geodesic \(\gamma\), the map \(\phi_\gamma\) of the Riemann foliation defined by
\[
N_\tau:TM\to TM,\qquad N_\tau(v)=\tau v,\qquad
\phi_\gamma:\mathbb C\to TM,\qquad \phi_\gamma(\sigma+i\tau)=N_\tau(\gamma'(\sigma))
\]
is holomorphic. The total space \((TM,J)\) equipped with an adapted complex structure is called an \emph{entire Grauert tube}.
\end{defn}
Thus, an adapted complex structure can be understood as a complex structure in which geodesics extend to complex curves. An entire Grauert tube is known to be a Stein manifold with the squared norm function \(\rho(v):=\|v\|_g^2\) as a strictly plurisubharmonic exhaustion. Moreover, entire Grauert tubes are strongly expected to be affine varieties, but this remains conjectural \cite{B82,D85}.

The case in which the base metric is Zoll is a model case for entire Grauert tubes. Indeed, every standard CROSS admits an entire Grauert tube \cite{PW91}. Since all geodesics are periodic with a common period \(l\), each leaf of the Riemann foliation can be viewed as a map \(\phi_\gamma:\mathbb C/l\mathbb Z\to TM\), whose domain is a genus-zero curve with two punctures.

Intuitively, we may therefore try to compactify the entire tangent bundle \(TM\) by filling in these two punctures on each leaf. In fact, even for a general Zoll manifold without a complex structure, this construction works as follows. The geodesic flow on the unit tangent bundle \(UM\) is periodic with a common period, so it defines a free \(S^1\)-action on \(UM\). We may therefore consider \(D\simeq UM/S^1\), which can be interpreted as the set of oriented images of geodesics. Classically, \(D\) is called the \emph{space of geodesics} \cite{B12}. For a leaf \(\phi_\gamma\) of the Riemann foliation corresponding to a geodesic, we may then attempt a smooth compactification by identifying its ends, where the imaginary part tends to \(+\infty\) and \(-\infty\), with the points of \(D\) corresponding to the forward and reverse orientations of the geodesic, respectively. See \cite[Section~2]{LS26} for a precise construction.
\begin{prop}\label{comp}
Let \((M,g)\) be a Zoll manifold, and let \(D\) be its space of geodesics. Then there exists a smooth compactification \(X:=TM\sqcup D\) such that the inclusion \(TM\hookrightarrow X\) is a smooth open embedding and \(D\) is a smooth submanifold of real codimension \(2\).
\end{prop}
If the Zoll manifold admits an adapted complex structure on the entire tangent bundle, this compactification \(X\) has the following properties.
\begin{prop}\label{fano}
Let \((M,g)\) be a Zoll manifold of type \((a,n)\) admitting an adapted complex structure on the entire tangent bundle, and let \(X:=TM\sqcup D\) be the compactification in Proposition~\ref{comp}.
\begin{enumerate}
\item \(X\) is a Fano manifold with Picard number \(1\) if \(a\neq 2\), and \(2\) if \(a=2\). The divisor \(D\) is ample, and \(-K_X=\frac{a(n+1)}{2}D\) \cite[Theorem~3.2]{S25}.
\item Each leaf of the Riemann foliation extends holomorphically to \(\phi_\gamma:\mathbb P^1\to X\), meeting \(D\) transversely at the two added points. We call the extended map \(\phi_\gamma\), or its image, a leaf \cite{BL18,LS26,S25}.
\item The direction-reversing map \(N_{-1}:TM\to TM\), \(N_{-1}(v)=-v\), extends to an antiholomorphic involution \(\tau:X\to X\), which has no fixed points on \(D\) \cite{BL18,LS26,S25}.
\item If \(X\) is biholomorphic to the standard ambient Fano variety of a CROSS, then \(M\) is isometric to the standard CROSS, up to constant rescaling \cite[Proposition~4.1]{S25}.
\end{enumerate}
\end{prop}

\subsection{Proof of Theorem~\ref{main}}
Theorem~\ref{main} was already proved for \((a,n)=(2,n)\), namely the \(\mathbb{CP}^n\) case, in \cite[Main theorem]{LS26}. As mentioned in the introduction, the Blaschke conjecture has already been resolved for \((a,n)=(a,1)\) and \((1,n)\), namely the \(S^a\) and \(\mathbb{RP}^n\) cases. Henceforth, let \((M,g)\) be a Blaschke manifold of type \((a,n)=(4,n)\) or \((8,2)\), and let \(X:=TM\sqcup D\) be the Fano compactification given in Proposition~\ref{fano}. Since \(M\subset X\) is totally real, we have a natural identification \(T_pM^{\mathbb C}\simeq T_pX\) for every \(p\in M\). Accordingly, we will often regard the great quadric bundle $Z_p$ as a submanifold of \(\mathbb P(T_pX)\).

\noindent \textbf{Summary of proof.} By Proposition~\ref{fano}\,\textup{(iv)}, it suffices to show that \(X\) is biholomorphic to \(\operatorname{Gr}(2,2n+2)\) or \(E_6/P_1\), respectively. For a general point \(p\in M\), every direction in \(Z_p\) is tangent to a minimal rational curve (Proposition~\ref{break}). Since \(Z_p\) has the same dimension as the VMRT (Proposition~\ref{analytic}), it is an irreducible component of the VMRT fiber associated with a single covering minimal rational component \(\mathcal K\), and is therefore a complex submanifold (Proposition~\ref{Zalg}). Theorem~\ref{hopfiso} then implies that \(Z_p\) is standard. Moreover, this choice of VMRT component along \(M\) propagates consistently along each leaf \(\phi_\gamma\) via the maximal slope subbundle of \(\phi_\gamma^*TX\), yielding the standardness of the VMRT at a general point of \(X\) (Propositions~\ref{split} and~\ref{Xrigid}). This allows us to identify \(X\) by the Hwang--Mok recognition theorem (Theorem~\ref{HM}).

\begin{prop}\label{break}
For a general point $p\in M$, let $Z_p\subset\mathbb P(T_pX)$ be the great quadric bundle at $p$. Then every direction $[\xi]\in Z_p$ is the tangent direction at $p$ of a minimal rational curve on $X$.
\end{prop}
\begin{proof}
Choose a general point $p\in M\subset X$. By \cite[Proposition~6.2]{BKK22}, there exists a dense Zariski open subset $U\subset X$ such that every rational curve $C\subset X$ with $D\cdot C=1$ passing through a point $p\in U$ is smooth at $p$. Since $M$ is Zariski dense in $X$, we may choose $p\in M\cap U$. Shrinking $U$ if necessary, we also avoid the loci of the finitely many non-dominating families of rational curves of $D$-degree one.

Fix $q\in B_p$, and normalize the common length of the closed geodesics to $2\pi$. Let $\gamma$ be a unit-speed closed geodesic with $\gamma(0)=p$ and $\gamma(\pi)=q$, and let $\phi_\gamma$ be the corresponding leaf. Set $V:=(\mathcal V_p)_q\simeq\mathbb R^a$ and $V^{\mathbb C}:=V\otimes_{\mathbb R}\mathbb C\simeq\mathbb C^a$. The compactified leaf meets $D$ transversely at the two points $0_\gamma$ and $\infty_\gamma$, so its $D$-degree is two.

With the convention $\phi_\gamma(w)=N_\tau(\gamma'(\sigma))$ for $w=\sigma+i\tau$, introduce the coordinate $z=(e^{iw}-1)/(e^{iw}+1)$ on the compactified leaf, and denote the resulting map by $f_\gamma:\mathbb P^1\to X$. The points $p$ and $q$ correspond to $w=0$ and $w=\pi$, while $\infty_\gamma$ and $0_\gamma$ correspond to the ends $\operatorname{Im}w\to+\infty$ and $\operatorname{Im}w\to-\infty$, respectively. Their $z$-coordinates are therefore $0,\infty,-1,1$. Moreover, $f_\gamma'(0)=-2i\gamma'(0)$, and hence
\begin{equation}\label{leaf}
\operatorname{cr}(0,\infty;1,-1)=-1,
\qquad
g_{\mathbb C,p}\bigl(f_\gamma'(0),f_\gamma'(0)\bigr)=-4.
\end{equation}
Both values are independent of the choice of $\gamma$ through $p$ and $q$.

Let $\beta$ be the common curve class of the compactified leaves, and let $H_{p,q}\subset\overline M_{0,4}(X,\beta)$ be the Zariski closure of the stable maps $[\mathbb P^1,0,\infty,1,-1;f_\gamma]$ as $\gamma$ varies. Write a member of $H_{p,q}$ as $[C,x_p,x_q,d_1,d_2;f]$. By closedness of the evaluation conditions, every member satisfies $f(x_p)=p$, $f(x_q)=q$, and $f(d_1),f(d_2)\in D$. Let $H_{p,q}^{\mathrm{sm}}$ denote the locus where the source $C$ is a smooth $\mathbb P^1$.

The forgetful morphism $\operatorname{st}:H_{p,q}\to\overline M_{0,4}$, which forgets the map and stabilizes the marked source, is constant with value $-1$ by \eqref{leaf} and the definition of $H_{p,q}$. On $H_{p,q}^{\mathrm{sm}}$, normalize the source coordinate by $x_p=0$, $x_q=\infty$, and $d_1=1$. Then $d_2=-1$. The expression $g_{\mathbb C,p}(f'(0),f'(0))$ is algebraic on this locus, so the second identity in \eqref{leaf} also holds throughout $H_{p,q}^{\mathrm{sm}}$. In particular, $f'(0)\neq0$, and the tangent evaluation $\tau:H_{p,q}^{\mathrm{sm}}\to\mathbb P(T_pX)$, defined by $\tau(f)=[f'(0)]$, is a morphism.

For the actual leaves, the tangent directions at $p$ form $\mathbb P(V)\simeq\mathbb{RP}^{a-1}$, which is Zariski dense in $\mathbb P(V^{\mathbb C})\simeq\mathbb P^{a-1}$. Since the condition $\tau(f)\in\mathbb P(V^{\mathbb C})$ is algebraic and holds for every leaf, it holds throughout $H_{p,q}^{\mathrm{sm}}$. Consequently, $\overline{\tau(H_{p,q}^{\mathrm{sm}})}=\mathbb P(V^{\mathbb C})$.

Consider the Zariski closure of the graph, $\Gamma:=\overline{\{(h,\tau(h)):h\in H_{p,q}^{\mathrm{sm}}\}}\subset H_{p,q}\times\mathbb P(V^{\mathbb C})$. The projectivity of the stable-map moduli space implies that the second projection of $\Gamma$ has closed image. Since this image contains $\tau(H_{p,q}^{\mathrm{sm}})$, it equals $\mathbb P(V^{\mathbb C})$. Thus, for every $[v]\in\mathbb P(V^{\mathbb C})$, there is a stable map $f\in H_{p,q}$ such that $(f,[v])\in\Gamma$.

Fix such a pair with $g_{\mathbb C,p}(v,v)=0$. If $f\in H_{p,q}^{\mathrm{sm}}$, then $[v]=\tau(f)$, contradicting \eqref{leaf}. Hence the source of $f$ is reducible. Nevertheless, its stabilized four-pointed source lies in the interior $M_{0,4}\subset\overline M_{0,4}$, since $\operatorname{st}(f)=-1$. We now examine its components.
\begin{enumerate}
\item There is no component of $D$-degree two.\\
Otherwise, all other components would be contracted, since $D$ is ample and the total degree is two. As the dual graph is a tree, every connected component of the complement of the degree-two component would be a contracted tail. Stability forces each such tail to carry at least two markings. All markings on a contracted tail have the same image, so $p\neq q$ and $p,q\notin D$ imply that its markings must be $d_1,d_2$. The stabilization would then have the partition $\{x_p,x_q\}\mid\{d_1,d_2\}$, giving a boundary value of the cross-ratio, a contradiction. Thus there are exactly two nonconstant components, both of $D$-degree one.

\item The component containing $x_p$ is not contracted.\\
Suppose otherwise, and let $G$ be the maximal connected union of contracted components containing $x_p$. Since $G$ maps to $p$, it contains none of the other markings. If $G$ has $k$ components and meets its complement at $b$ nodes, stability gives $2(k-1)+b+1\geq3k$, hence $b\geq k+1$. Each branch attached to $G$ contains a nonconstant component, and there are only two such components. Since the dual graph is a tree, $b\leq2$. Therefore $k=1$ and $b=2$: the contracted component $E_0$ containing $x_p$ joins two degree-one components $E_1,E_2$.

Any additional contracted subtree must be a tail attached to this configuration. As above, such a tail must contain both $d_1,d_2$ and is excluded by the constant cross-ratio. Hence the source consists precisely of $E_0,E_1,E_2$. The markings $d_1,d_2$ lie on $E_1\cup E_2$. Each $E_i$ passes through $p\notin D$, and the nonzero section $(f|_{E_i})^*s_D$ has degree one, so it cannot vanish at two distinct marked points. Thus $d_1,d_2$ lie on different components. The marking $x_q$ also lies on one of these two components. Stabilization then gives either $\{d_1,x_q\}\mid\{d_2,x_p\}$ or $\{d_1,x_p\}\mid\{d_2,x_q\}$, again contradicting the constant cross-ratio.
\end{enumerate}

Let $E_p$ be the component containing $x_p$, and set $C_p:=f(E_p)$. Since $E_p$ has $D$-degree one, $f|_{E_p}$ is the normalization map of $C_p$, and $D\cdot C_p=1$. By our choice of $p$, the curve $C_p$ is smooth at $p$, so $df_{x_p}\neq0$.

To identify its tangent direction, take a one-parameter specialization in $\Gamma$ to $(f,[v])$ whose general member belongs to the graph of $\tau$. After a finite base change if necessary, represent this specialization by a family of stable maps. Near the marked point $x_p$, the family of source curves is smooth, so the map can be written locally as $F(t,u)$ with $F(t,0)=p$ and $\partial_uF(0,0)\neq0$. Continuity of the differential gives $[v]=\lim_{t\to0}[\partial_uF(t,0)]=[T_pC_p]$.

Since $D$ is ample, $C_p$ has the smallest possible positive $D$-degree. By our choice of $p$, it belongs to a dominating family, which is unsplit because its members have $D$-degree one. Thus $C_p$ is a minimal rational curve. Letting $q$ vary over $B_p$ proves the assertion for every direction in $Z_p$.
\end{proof}

\begin{prop}\label{analytic}
For any $p\in M$, $Z_p\subset \mathbb P(T_pX)$ is a real-analytic submanifold of real dimension $\dim_{\mathbb R}Z_p=a(n+1)-4$.
\end{prop}
\begin{proof}
We first compute the real dimension of $Z_p$. Since $\dim_{\mathbb R}Q^{a-2}=2a-4$ and $\dim_{\mathbb R}B_p=an-a$, we have $\dim_{\mathbb R}Z_p=\dim_{\mathbb R}B_p+\dim_{\mathbb R}Q^{a-2}=an+a-4=a(n+1)-4=2d$, where $d:=\frac{a(n+1)}{2}-2$. Moreover, a metric admitting an adapted complex structure is real-analytic by \cite{LS91,S91}. After normalizing the period of every closed geodesic to be $2\pi$, $Z_p$ is characterized by
\begin{align*}
Z_p=\{[a+ib]\in\mathbb P(T_pX):&a,b\in T_pM,\ \|a\|_g=\|b\|_g=1,\ g(a,b)=0,\\
&\exp_p(\pi a)=\exp_p(\pi b)=q\text{ for some }q\in B_p\}.
\end{align*}
Since all the functions and equations appearing in this characterization are real-analytic, and $Z_p\subset\mathbb P(T_pX)$ is a smooth submanifold, $Z_p$ is real-analytic.
\end{proof}

\begin{prop}\label{Zalg}
There exists a covering minimal rational component $\mathcal K$ on $X$ such that, for a general point $p\in M$, the great quadric bundle $Z_p$ is an irreducible component of the VMRT fiber $\mathcal C_p$ associated with $\mathcal K$. In particular, $Z_p\subset\mathbb P(T_pX)$ is a complex submanifold, and hence Theorem~\ref{hopfiso} applies.
\end{prop}
\begin{proof}
For any minimal rational curve $C$ on $X$, note that $-K_X\cdot C=\frac{a(n+1)}{2}$ by \cite{S25}. Thus, setting $d:=\frac{a(n+1)}{2}-2$, \cite[Corollary~IV.2.9]{K13} gives, for a general member $f:\mathbb P^1\to X$ of a minimal rational component,
\[
f^*T_X\simeq\mathcal O(2)\oplus\mathcal O(1)^{\oplus d}\oplus\mathcal O^{\oplus(an-d-1)}.
\]
Moreover, by \cite[Theorem~1]{HM04}, for a general point $x\in X$, the tangent map $\tau_x:\mathcal K_x\to\mathcal C_x\subset\mathbb P(T_xX)$ is birational. Hence, by deformation theory, $\dim_{\mathbb C}\mathcal C_x=\dim_{\mathbb C}\mathcal K_x=(d+2)-2=d$. Thus $Z_p$ and $\mathcal C_p$ have the same real dimension.

We now find a minimal rational component $\mathcal K$ satisfying the proposition. Since all rational curves under consideration have $D$-degree one, only finitely many irreducible families occur \cite[Theorem~V.1.6.2]{K13}. Avoid the loci of all non-covering families and choose the general points $p\in M$ appearing in Proposition~\ref{break}. Let $\mathfrak Z$ denote the total space of the bundle $Z_p\to\mathfrak Z\to M$ over all of $M$. Since $M$ and the fibers $Z_p$ are connected, $\mathfrak Z$ is connected. For each covering minimal rational component $\mathcal K_i$, let $\mathcal C_{i,p}$ be its VMRT at $p$, and let $\mathfrak C_i\subset\mathbb P(TX)$ denote the corresponding total VMRT. By Proposition~\ref{break}, over the dense set of general points chosen above, every point of $\mathfrak Z$ belongs to one of the $\mathfrak C_i$. Since $\bigcup_i\mathfrak C_i$ is closed, it follows that $\mathfrak Z\subset\bigcup_i\mathfrak C_i$.

If $\mathfrak Z\cap\mathfrak C_i$ had empty interior in $\mathfrak Z$ for every $i$, finitely many such subsets could not cover $\mathfrak Z$. Hence some $\mathfrak C_i$ contains a nonempty open subset of $\mathfrak Z$. Since $\mathfrak C_i$ is algebraic, its local defining holomorphic functions restrict to real-analytic functions on the connected real-analytic manifold $\mathfrak Z$. They vanish on a nonempty open subset, and hence vanish identically. Therefore $\mathfrak Z\subset\mathfrak C_i$. Fix $\mathcal K:=\mathcal K_i$. Then, for every general $p\in M$, $Z_p\subset\mathcal C_p$.

Writing $\mathcal C_p$ as the finite union of its irreducible components and applying the same argument to $Z_p$, we obtain an irreducible component $C_p\subset\mathcal C_p$ such that $Z_p\subset C_p$. Since $Z_p\subset C_p$ has the same real dimension as $C_p$, the intersection $Z_p\cap C_p^{\mathrm{reg}}$ is nonempty and open in $C_p^{\mathrm{reg}}$. It is also closed in $C_p^{\mathrm{reg}}$ because $Z_p$ is compact. Since $C_p^{\mathrm{reg}}$ is connected, we obtain $C_p^{\mathrm{reg}}\subset Z_p$, and hence $Z_p=C_p$. Thus $Z_p$ is an irreducible component of $\mathcal C_p$. In particular, $Z_p$ is a complex submanifold of $\mathbb P(T_pX)$, and Theorem~\ref{hopfiso} applies.
\end{proof}
\begin{prop}\label{split}
For every leaf \(\phi_\gamma:\mathbb P^1\to X\), we have \(\phi_\gamma^*TX\simeq\mathcal O(2)^{\oplus a}\oplus\mathcal O(1)^{\oplus(an-a)}\).
\end{prop}
\begin{proof}
Let \(N:=N_{\phi_\gamma/X}\), and choose \(p=\gamma(0)\) and \(q=\gamma(\pi)\in B_p\). Normal Jacobi fields along \(\gamma\) admit holomorphic extensions along the compactified leaf; see the parallel vector fields in \cite{LS26}. By \cite[proof of Lemma~7]{BL18}, we may choose their holomorphic normal extensions \(\xi_1^{1,0},\ldots,\xi_{an-1}^{1,0}\) with the following properties:
\begin{enumerate}
\item The fields \(\xi_1^{1,0},\ldots,\xi_{a-1}^{1,0}\) vanish at \(p\) and \(q\), and \(\xi_a^{1,0},\ldots,\xi_{an-1}^{1,0}\) vanish at \(p\).
\item At a general point of \(\mathbb P^1\), these fields form a basis of \(N\).
\end{enumerate}
The vanishing in (1) gives sections of \(N(-p-q)\) and \(N(-p)\), respectively, and hence a bundle map \(F:\mathcal O(p+q)^{\oplus(a-1)}\oplus\mathcal O(p)^{\oplus(an-a)}\to N\), which is generically an isomorphism by (2).

Since \(-K_X=\frac{a(n+1)}{2}D\) and the leaf meets \(D\) transversely at two points, we have \(\deg N=-K_X\cdot\phi_\gamma-2=an+a-2\). The source of \(F\) has the same degree. Thus \(\det F\) is a nonzero section of a degree-zero line bundle on \(\mathbb P^1\), so it is nowhere vanishing. Consequently, \(F\) is an isomorphism, and \(N\simeq\mathcal O(2)^{\oplus(a-1)}\oplus\mathcal O(1)^{\oplus(an-a)}\).

Finally, the tangent-normal sequence \(0\to\mathcal O(2)\to\phi_\gamma^*TX\to N\to0\) splits because
\[\operatorname{Ext}^1(N,\mathcal O(2))=H^1(\mathbb P^1,N^\vee\otimes\mathcal O(2))=0.\]
This gives the asserted splitting.
\end{proof}
Finally, we prove that the compactification $X=TM\sqcup D$ is standard. Together with Proposition~\ref{fano}(4), this proves Theorem~\ref{main}.
\begin{prop}\label{Xrigid}
For the minimal rational component $\mathcal K$ obtained in Proposition~\ref{Zalg}, the VMRT fiber $\mathcal C_z$ at a general point $z\in X$ is projectively equivalent to the standard model. In particular, $X$ is biholomorphic to the corresponding standard compactification $X=TM\cup D$ listed in Table~\ref{amb}.
\end{prop}
\begin{proof}
We first explain the rough idea. Proposition~\ref{Zalg} gives, over general points of $M$, a section $s_M$ of the finite \'etale cover whose fiber over $z\in X$ parametrizes the connected components of $\mathcal K_z$. We extend this section consistently over $X$, thereby showing that the covering map has singleton fibers. Its standardness then extends to general points of $X$ by the Zariski density of $M$ and the local closedness of projective linear group orbits in the Hilbert scheme. We may then apply the following VMRT recognition theorem of Hwang--Mok \cite[Main Theorem]{M08}.

\begin{thm}\label{HM}
Let $S$ be a rational homogeneous manifold of Picard number $1$ which is either a Hermitian symmetric space or a Fano contact homogeneous manifold. For a reference point $o\in S$, let $\mathcal C_o\subset\mathbb P T_oS$ denote the variety of minimal rational tangents on $S$. Let $X$ be a Fano manifold of Picard number $1$ and $\mathcal K$ a minimal rational component on $X$. Suppose that the variety of $\mathcal K$-rational tangents $\mathcal C_x\subset\mathbb P T_xX$ at a general point $x\in X$ is isomorphic to $\mathcal C_o\subset\mathbb P T_oS$ as a projective subvariety. Then $X$ is biholomorphic to $S$.
\end{thm}

We now formalize this idea. Let $\operatorname{ev}:\mathcal U\to X$ be the evaluation morphism of the normalized universal family. Choose a dense Zariski-open subset $X^\circ\subset X\setminus D=TM$ over which $\operatorname{ev}$ is smooth with fibers of constant dimension, and consider the Stein factorization
\[
\mathcal U^\circ\xrightarrow{h}\widetilde X^\circ\xrightarrow{\pi}X^\circ.
\]
The space $\widetilde X^\circ$ parametrizes the connected components of the fibers of $\operatorname{ev}$. The morphism $\pi$ is finite \'etale, and for a general point $z\in X$ the fiber $\pi^{-1}(z)$ corresponds to the connected components of $\mathcal K_z$. Recall also that the tangent map $\tau_z:\mathcal K_z\to\mathcal C_z$ is birational. By Proposition~\ref{Zalg}, the distinguished component $Z_p\subset\mathcal C_p$ therefore determines, for $p\in M^\circ:=M\cap X^\circ$, a section
\[
s_M:M^\circ\longrightarrow\widetilde X^\circ,\qquad \pi\circ s_M=\operatorname{id}.
\]

We extend $s_M$ to $X^\circ$. Since every point of $X\setminus M$ lies on a unique compactified leaf $\phi_\gamma$, we first extend it leafwise. By Proposition~\ref{split}, recall that
\[
T_X|_{\phi_\gamma}\simeq\mathcal O(2)^{\oplus a}\oplus\mathcal O(1)^{\oplus(an-a)}.
\]
The maximal slope subbundle $\mathcal O(2)^{\oplus a}$ is induced by the Jacobi fields having a conjugate point. Along the real geodesic $\gamma\subset M$, its fibers are precisely the fibers of the complexified Blaschke bundle $V^{\mathbb C}$. Write $F_\gamma:=\mathcal O(2)^{\oplus a}$ for this subbundle. By \cite[Proposition~7.4]{S26}, the metric $g$ admits a single-valued meromorphic extension $g_{\mathbb C}$ along $\phi_\gamma\cap TM$. Since its restriction to $F_\gamma$ is nondegenerate along $\gamma$, its isotropic lines define a holomorphic quadric bundle $\mathcal Q\subset\mathbb P(F_\gamma)$ away from a closed discrete subset $S_\gamma\subset\phi_\gamma\cap X^\circ$, disjoint from $\gamma$. Along $\gamma\cap X^\circ$, each fiber of $\mathcal Q$ is contained in the corresponding $Z_p$. By the identity theorem, the inclusion in the VMRT family extends throughout $(\phi_\gamma\cap X^\circ)\setminus S_\gamma$.

For a general point \(p\in M^\circ\), choose \(q\in B_p\) so that the corresponding quadric \(Q_{p,q}\subset Z_p\) is contained in no other irreducible component of \(\mathcal C_p\). Such choices form an open dense subset of \(B_p\): for each other component, the set of \(q\) whose quadric is contained in it is closed and has empty interior, since otherwise that component would contain a nonempty open subset of \(Z_p\). Along a leaf through such $p,q$, the local sheet selected by $s_M$ continues, away from $S_\gamma$, to a component containing the corresponding fiber of $\mathcal Q$, by holomorphicity and the identity theorem. Continuation around any loop in $(\phi_\gamma\cap X^\circ)\setminus S_\gamma$ must therefore return to the original sheet, by uniqueness at $p$. The resulting section extends uniquely across $S_\gamma$ because $\pi$ is \'etale. The resulting section agrees with \(s_M\) on every real arc in \(\gamma\cap X^\circ\), since continuation from any such arc likewise yields a component containing \(Q_{p,q}\) at \(p\). These choices give a dense family of leaves.

For any remaining leaf meeting \(X^\circ\), its real geodesic also meets \(M^\circ\), by the identity theorem. Any fixed path or loop in \(\phi_\gamma\cap X^\circ\), with initial point on \(M^\circ\), can be deformed to nearby leaves of the dense family. Continuous dependence of path lifting then preserves both triviality of the selected monodromy and agreement with \(s_M\). Thus the leafwise sections extend consistently to all leaves meeting \(X^\circ\). The continuous dependence of the leafwise construction, together with uniqueness of the leaf through each point of \(X\setminus M\), gives a continuous section \(s_X:X^\circ\to\widetilde X^\circ\) extending \(s_M\).

It follows that, for a general point $z\in X$, the space $\mathcal K_z$ is smooth and connected, hence irreducible, and so is its image $\mathcal C_z$ under the tangent map. Thus, for general $p\in M$, we have $\mathcal C_p=Z_p$, which is projectively equivalent to the standard model $S=\mathbb P^1\times\mathbb P^{2n-1}$ or $\mathbb S_5$. Choose a nonempty Zariski-open subset $U\subset X^\circ$ on which $TX$ is trivial and the VMRT family is flat with reduced fibers. Since projective linear group orbits in the Hilbert scheme are locally closed, the locus of points $z\in U$ for which $\mathcal C_z$ is projectively equivalent to $S$ is locally closed. This locus contains the general points of $M\cap U$ and is therefore Zariski dense in $U$, hence open. The VMRT at a general point of $X$ is consequently projectively equivalent to $S$, and the recognition theorem above gives the desired biholomorphism.
\end{proof}

\bibliography{KSong}
\bibliographystyle{plain}

\vskip 3mm
\noindent
Department of Mathematics, Rutgers University, Piscataway, NJ 08854-8019.

\noindent
{\it Email:} ks1951@rutgers.edu

\end{document}